\documentclass[11pt]{amsart}
\usepackage{graphicx,amssymb,upref}

\usepackage{color}
\usepackage{overpic}
\usepackage{algorithm}
\usepackage{algorithmic}

\newtheorem{theorem}{Theorem}[section]

\newtheorem{lemma}[theorem]{Lemma}

\newtheorem*{conjecture}{Conjecture}

\theoremstyle{definition}
\newtheorem{question}[theorem]{Question}

\theoremstyle{remark}
\newtheorem{remark}{Remark}

\definecolor{MyGreen}{rgb}{0, 0.5, 0}
\definecolor{MyOrange}{cmyk}{0, 0.6,1.0,0}
\newcommand{\Red}[1]{{\color{red}#1}}
\newcommand{\Blue}[1]{{\color{blue}#1}}
\newcommand{\Green}[1]{{\color{MyGreen}#1}}
\newcommand{\Orange}[1]{{\color{MyOrange}#1}}
\newcommand{\Cyan}[1]{{\color{cyan}#1}}
\newcommand{\Magenta}[1]{{\color{magenta}#1}}

\title{Cosmetic banding on knots and links}

\author{Kazuhiro Ichihara}
\address{Department of Mathematics, College of Humanities and Sciences, Nihon University, 3-25-40 Sakurajosui, Setagaya-ku, Tokyo 156-8550, Japan.}
\email{ichihara@math.chs.nihon-u.ac.jp}
\thanks{Ichihara is partially supported by JSPS KAKENHI Grant Number 26400100.}

\author[{In Dae Jong with Hidetoshi Masai}]{In Dae Jong \\with an appendix by Hidetoshi Masai}
\address{Department of Mathematics, Kindai University, 3-4-1 Kowakae, Higashiosaka City, Osaka 577-0818, Japan} 
\email{jong@math.kindai.ac.jp}

\address{Graduate School of Mathematical Sciences, The University of Tokyo, 3-8-1 Komaba Meguro-ku Tokyo 153-8914, Japan}
\email{masai@ms.u-tokyo.ac.jp}
\thanks{Masai is partially supported by JSPS Research Fellowship for Young Scientists.}

\dedicatory{Dedicated to Professor Taizo Kanenobu on the occasion of his 60th birthday}


\begin{document}

\begin{abstract}
We present various examples of cosmetic bandings on knots and links, 
that is, bandings on knots and links leaving their types unchanged. 
As a byproduct, we give a hyperbolic knot which admits exotic chirally cosmetic surgeries yielding hyperbolic manifolds. 
This gives a counterexample to a conjecture raised by Bleiler, Hodgson and Weeks. 
\end{abstract}

\maketitle
\tableofcontents

\section{Introduction}\label{sec:intro}

In this paper, we call the following operation on a link a \textit{banding}\footnote{The operation is sometimes called a band surgery, a bund sum (operation), or a hyperbolic transformation in a variety of contexts. 
In this paper, referring to \cite{Bleiler}, we use the term banding to clearly distinguish it from a Dehn surgery on a knot.} on the link. 
For a given link $L$ in the 3-sphere $S^3$ and an embedding $b \colon I \times I \to S^3$ such that $b ( I \times I ) \cap L = b ( I \times \partial I )$, where $I$ denotes a closed interval, we obtain a (new) link as $ ( L - b ( I \times \partial I ) ) \cup b ( \partial I \times I )$. 
We call the link so obtained \emph {the link obtained from $L$ by a banding 
along the band $b$}. 

This operation would be first studied in \cite{Lickorish}, and then, played very important roles in various scenes in Knot theory, and also in much wider ranges, for example, relation to DNA topology (see \cite{IshiharaShimokawaVazquez} for example). 
Also see \cite{AbeKanenobu, Kanenobu10, Kanenobu12, KanenobuMoriuchi} for some of recent studies. 

\begin{remark}
On performing a banding, it is often assumed the compatibility of orientations of the original link and the obtained link, but in this paper, we do not assume that. 
Also note that this operation for a knot yielding a knot appears as the $n=2$ case of the $H(n)$-move on a knot, which was introduced in \cite{HosteNakanishiTaniyama}. 
See \cite{KanenobuMiyazawa} for a study of $H(2)$-move on knots for example. 
\end{remark}

In general, for an operation on knots and links, a fundamental question 
must be when the operation preserves the type of a knot and link. 
For example, for the crossing change operation, this question is related to the famous Nugatory Crossing Conjecture \cite[Problem 1.58]{Kirby}. 

Let us consider this question for a banding in this paper. 
Precisely we say that a banding on a link $L$ in $S^3$ is 
\textit{cosmetic} if the link $L'$ obtained by a (non-trivial) banding on $L$ 
is equivalent to $L$, in other words, 
there exists a self-homeomorphism of $S^3$ which takes $L$ to $L'$. 
If the self-homeomorphism is orientation-preserving (resp.~orientation-reversing), 
then we call the banding is a \textit{purely} cosmetic (resp.~\textit{chirally} cosmetic). 

There always exists a ``trivial'' banding on a link which is cosmetic, 
that is when the band is half-twisted and parallel to the link. 
However, also some other ``non-trivial'' examples are known. 
In this paper, we analyze and describe mechanisms of known examples, and present more examples of cosmetic bandings. 
Detailed contents are as follows. 

In Section~\ref{sec:BandAndSurgery}, we give a brief review on 
a relationship between a banding and a Dehn surgery. 
We also introduce some terminologies used in this paper. 

In Section~\ref{sec:4-move}, we give examples of cosmetic bandings 
realized by a $4$-move operation. 
In particular, based on the examples given in \cite{Kanenobu11}, 
we present a family of infinitely many two-bridge links which admit cosmetic bandings. 
Actually we see that the family coincides with the set of two-bridge links with unlinking number one. 

In Section~\ref{sec:TorusKnot}, we analyze the example of a cosmetic banding 
given in \cite{Zekovic}, 
that is, a cosmetic banding on the torus knot of type $(2,5)$. 
We will describe why such a cosmetic banding can occur, 
and point out the example is quite sporadic. 

In Section~\ref{sec:9_27}, we see that the knot $9_{27}$ in the knot table admits a cosmetic banding, and reveal its mechanism. 
As a matter of fact, the example comes from the famous example related to the Cosmetic Surgery Conjecture 
given in \cite{BleilerHodgsonWeeks}. 
As a byproduct, we give a hyperbolic knot which admits exotic chirally cosmetic Dehn surgeries yielding hyperbolic 3-manifolds. 
This gives a counterexample to a conjecture raised by Bleiler, Hodgson and Weeks in \cite[Conjecture 2]{BleilerHodgsonWeeks}.

\begin{remark}
The known examples and our examples are all chirally ones. 
Therefore it remains open whether there exists a non-trivial purely cosmetic banding on a link. 
\end{remark}

\section{Banding and Dehn surgery}\label{sec:BandAndSurgery}

In this section, as a basic tool used throughout this paper, 
we describe a relationship between a banding on a link and a Dehn surgery on a knot. 

\subsection{Dehn surgery}

Let $K$ be a knot in a 3-manifold $M$ with the exterior $E(K)$ (i.e., the complement of an open tubular neighborhood of $K$). 
Let $\gamma$ be a slope (i.e., an isotopy class of non-trivial simple closed curves) on the boundary torus $\partial E(K)$. 
Then, the \textit{Dehn surgery on $K$ along $\gamma$} 
is defined as the following operation. 
Glue a solid torus $V$ to $E(K)$ such that a simple closed curve representing $\gamma$ bounds a meridian disk in $V$. 
We denote the obtained manifold by $K(\gamma)$.

It is said that the Dehn surgery along the meridional slope is the \textit{trivial} Dehn surgery. 
Also it is said that a Dehn surgery along a slope which is represented by a simple closed curve with single intersection point with the meridian is an \textit{integral} Dehn surgery. 

In the case where $M = S^3$, we have the well-known bijective correspondence between $\mathbb{Q} \cup \{ 1/0 \}$ and the slopes on $\partial E(K)$, which is given by using the standard meridian-longitude system for $K$. 
When the slope $\gamma$ corresponds to $r \in \mathbb{Q} \cup \{1/0\}$, then the Dehn surgery on $K$ along $\gamma$ is said to be 
the \textit{$r$-Dehn surgery} on $K$, or the \textit{$r$-surgery} on $K$ for brevity. 
We also denote the obtained manifold by $K(r)$.
In this case, we note that an integral Dehn surgery corresponds to an $n$-Dehn surgery with an integer $n$. 

\subsection{Montesinos trick}

We here recall the Montesinos trick originally introduced in \cite{Montesinos}. 
Let $\bar{M}$ be the double branched cover of $S^3$ branched along a link $L \subset S^3$. 
Let $K$ be a knot in $\bar{M}$, which is \textit{strongly invertible} with respect to the preimage $\bar{L}$ of $L$, that is, there is an orientation preserving involution of $\bar{M}$ with the quotient $M$ and the fixed point set $\bar{L}$ which induces an involution of $K$ with two fixed points. 
Then the manifold $K(\gamma)$ obtained by an integral Dehn surgery on $K$ is homeomorphic to the double branched cover along the link obtained from $L$ by a banding along the band appearing as the quotient of $K$ with the corresponding framing $r$. 
See \cite{BakerBuck, BakerBuckLecuona, Bleiler} for example.

\subsection{Cosmetic surgery conjecture} 

In view of the Montesinos trick, studying cosmetic bandings on links is directly related to the following well-known conjecture.

\medskip
\noindent
\textbf{Cosmetic Surgery Conjecture} (\cite[Conjecture 2]{BleilerHodgsonWeeks}, also see \cite[Problem 1.81(A)]{Kirby}): 
Two surgeries on inequivalent slopes are never purely cosmetic. 
\medskip

Here two slopes are called \emph{equivalent} if there exists a self-homeomorphism of the exterior of a knot $K$ taking one slope to the other, and two surgeries on $K$ along 
two slopes are called \emph{purely cosmetic} (resp.\ \emph{chirally cosmetic}\footnote{In \cite{BleilerHodgsonWeeks}, it is called \emph{reflectively cosmetic}.}) if there exists an orientation preserving (resp.\ reversing) homeomorphism between the pair of the manifolds obtained by the pair of the surgeries. 
We say that a pair of cosmetic surgeries on a knot is \emph{mundane} 
(resp.\ \emph{exotic}) if the surgeries are along equivalent (resp.\ inequivalent) slopes. 
These terminologies were introduced in \cite{BleilerHodgsonWeeks}. 

The conjecture suggests that exotic purely cosmetic bandings 
on links might not exist, or are quite hard to find if exist. 
On the other hand, it is known that the cosmetic surgery conjecture 
for ``chirally cosmetic'' case is not true. 
That is, there exist many knots admitting non-trivial chirally cosmetic surgeries. 
From such examples, we can have cosmetic bandings on links as we will see in the rest of this paper.

\subsection{Mundane and exotic banding}

According to the definitions of a mundane cosmetic surgery and 
an exotic cosmetic surgery, 
we define a mundane cosmetic banding and an exotic cosmetic banding as follows. 

It is well-known that a rational tangle is determined by the meridional disk in the tangle. 
The boundary of the meridional disk is parameterized by an element of $\mathbb{Q} \cup \{1/0\}$, 
called a \emph{slope} of the rational tangle. 
A rational tangle is said to be \emph{integral} if the slope is an integer or $1/0$. 
For brevity, we call an integral tangle with slope $n$ an \emph{$n$-tangle}. 

A banding can be regarded as an operation replacing a $1/0$-tangle into an $n$-tangle. 
Then we call this banding a \emph{banding with a slope} $n$, and the $3$-ball corresponds to the $n$-tangle the \emph{banding ball}. 

A cosmetic banding with slope $n$ is said to be \emph{mundane} when there exists a self-homeomorphism on the exterior of the banding ball taking the slope $1/0$ to $n$. 
A cosmetic banding which is not mundane is said to be \emph{exotic}.

\section{4-moves}\label{sec:4-move}

In this section, we introduce examples of chirally cosmetic bandings 
which come from a $4$-move operation. 
By these examples, one might say that examples of chirally cosmetic bandings are not difficult to find. 
However, all of examples of chirally cosmetic bandings introduced in this section are mundane. 

A \emph{$4$-move} on a link is a local change that involves 
replacing parallel lines by $4$ half-twists 
as shown in the left-hand side of Figure~\ref{fig:4-move0}. 
We may regard the local move 
given in the right-hand side of Figure~\ref{fig:4-move0} as a $4$-move also. 

\begin{figure}[!htb]
\centering
\begin{overpic}[width=.6\textwidth]{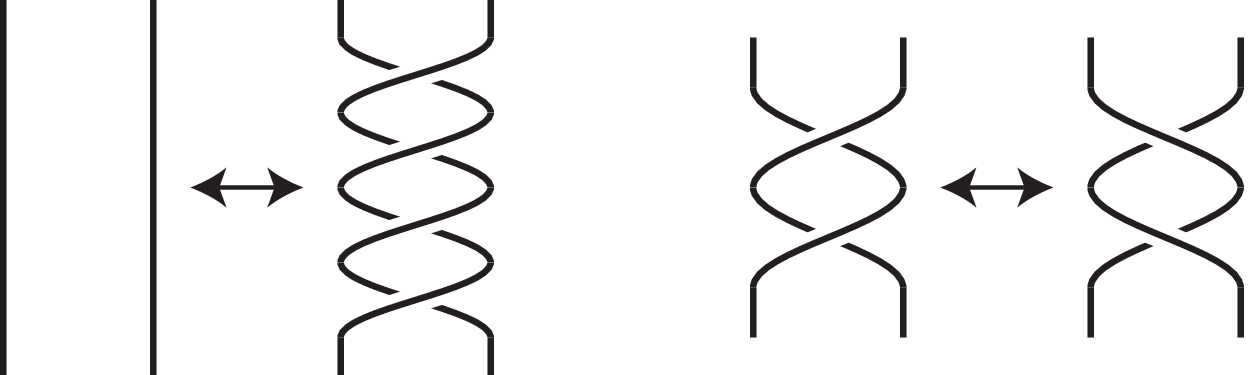}
\end{overpic}
\caption{$4$-moves}
\label{fig:4-move0}
\end{figure}

First, we show the following lemma. 

\begin{lemma}\label{lem:4-move}
A $4$-move on a link is realizable by a single banding. 
\end{lemma}
\begin{proof}
Apply the banding as shown in Figure~\ref{fig:4-move}. 
\begin{figure}[!htb]
\centering
\begin{overpic}[width=.4\textwidth]{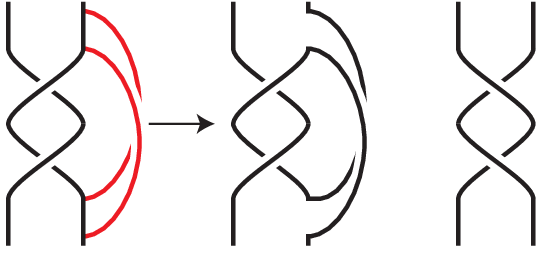}
\put(73,21){$=$}
\end{overpic}
\caption{The banding yields a $4$-move.}
\label{fig:4-move}
\end{figure}
\end{proof}

\subsection{Symmetric union}

By Lemma~\ref{lem:4-move}, 
a \emph{symmetric union} as shown in Figure~\ref{fig:SymUnion} 
admits a cosmetic banding. 
This fact was pointed out by Kanenobu~\cite[Section 11]{Kanenobu11}. 
Note that such symmetric unions contain some familiar knots, 
for example, a pretzel knot of type $(\pm 2, p, -p)$, 
more generally a Montesinos knot of type $(\pm \frac{1}{2}, R, -R)$. 
Furthermore the Kinoshita-Terasaka knot and the Conway knot are also contained, 
see \cite[Figure 15]{Kanenobu11}.

\begin{figure}[!htb]
\centering
\begin{overpic}[width=.4\textwidth]{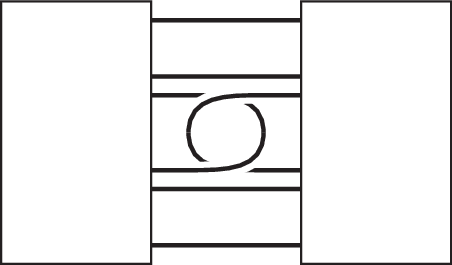}
\put(13,27){$T$}
\put(82,27){$T^*$}
\put(48.3,14.5){\rotatebox{-90}{$\cdots$}}
\put(48.3,52){\rotatebox{-90}{$\cdots$}}
\end{overpic}
\caption{The tangle $T^*$ is obtained from $T$ by mirroring along the vertical plane.}
\label{fig:SymUnion}
\end{figure}
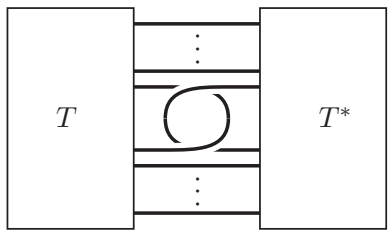

\subsection{Satellite knot/link}

By using Lemma~\ref{lem:4-move}, other examples are obtained as follows. 
For an amphicheiral knot in $S^3$, the Whitehead double with the blackboard framing 
is a knot admitting a cosmetic banding, 
and the $(2,2)$-cable is a $2$-component link admitting a cosmetic banding.

\subsection{Two-bridge link}

Further examples are given by unlinking number one two-bridge links as follows. 
Two-bridge links with unlinking number one were determined by Kohn~\cite{Kohn}. 
That is, the unlinking number of a two-bridge link $L$ is one if and only if 
there exist relatively prime integers $m$ and $n$ such that 
\[ L = S(2n^2, 2nm \pm 1) \] 
in the Schubert form. 
In the Conway form, this condition is equivalent to 
\[ L = C(a_0, a_1, \dots, a_k, \pm 2, -a_k, \dots, -a_1, -a_0) \, . \]
It is clear that such a two-bridge link admits 
a chirally cosmetic banding realized by a 4-move which changes 
\[ C(a_0, a_1, \dots, a_k, \pm 2, -a_k, \dots, -a_1, -a_0)\]  
into 
\[ C(a_0, a_1, \dots, a_k, \mp 2, -a_k, \dots, -a_1, -a_0) \, . \] 

These examples also can be explained by considering 
cosmetic Dehn surgeries on knots in lens spaces. 
Note that the double-branched cover of $S^3$ 
along a two-bridge link is a lens space. 
In~\cite{Matignon}, Matignon gave a complete list of 
cosmetic surgeries on non-hyperbolic knots in lens spaces. 
In particular, a Dehn surgery on a lens space $L(2m^2, 2mn-1)$ 
along the knot, $K_{m,n}$, the $(m,n)$-cable of the axis in $L(2m^2, 2mn-1)$ 
is contained in his list. 
Here $m$ and $n$ are positive and coprime integers with $2n \le m$. 
By taking the quotient with respect to a strong inversion, 
we can notice that the cosmetic surgery on $K_{m,n}$ 
yields a cosmetic banding on the two-bridge link $S(2m^2, 2mn-1)$. 
Furthermore this corresponds to the banding such that 
$C(a_0, a_1, \dots, a_k, -2, -a_k, \dots, -a_1, -a_0)$ changes to 
$C(a_0, a_1, \dots, a_k, 2, -a_k, \dots, -a_1, -a_0)$. 
Thus a cosmetic banding on a two-bridge link with unlinking number one 
corresponds to the cosmetic surgery on a lens space along a cable knot of a torus knot. 

\begin{remark}
In~\cite{Rong}, Rong determined cosmetic Dehn fillings on a Seifert fibered manifold 
with a torus boundary. 
Among the cosmetic surgeries and fillings contained in Matignon's list and Rong's list, 
the cosmetic surgery on a lens space $L(2m^2, 2mn-1)$ along $K_{m,n}$ due to Matignon 
is the only example yielding a cosmetic banding in the down-stair 
since the other cosmetic surgeries and fillings have no integral slopes. 
\end{remark}

\begin{remark}
It is easy to see that the chirally cosmetic banding realized by a $4$-move 
is mundane as shown in Figure~\ref{fig:mundane}. 
Thus, all of the cosmetic bandings introduced in this section are mundane. 
\end{remark}
\begin{figure}[!htb]
\centering
\begin{overpic}[width=.7\textwidth]{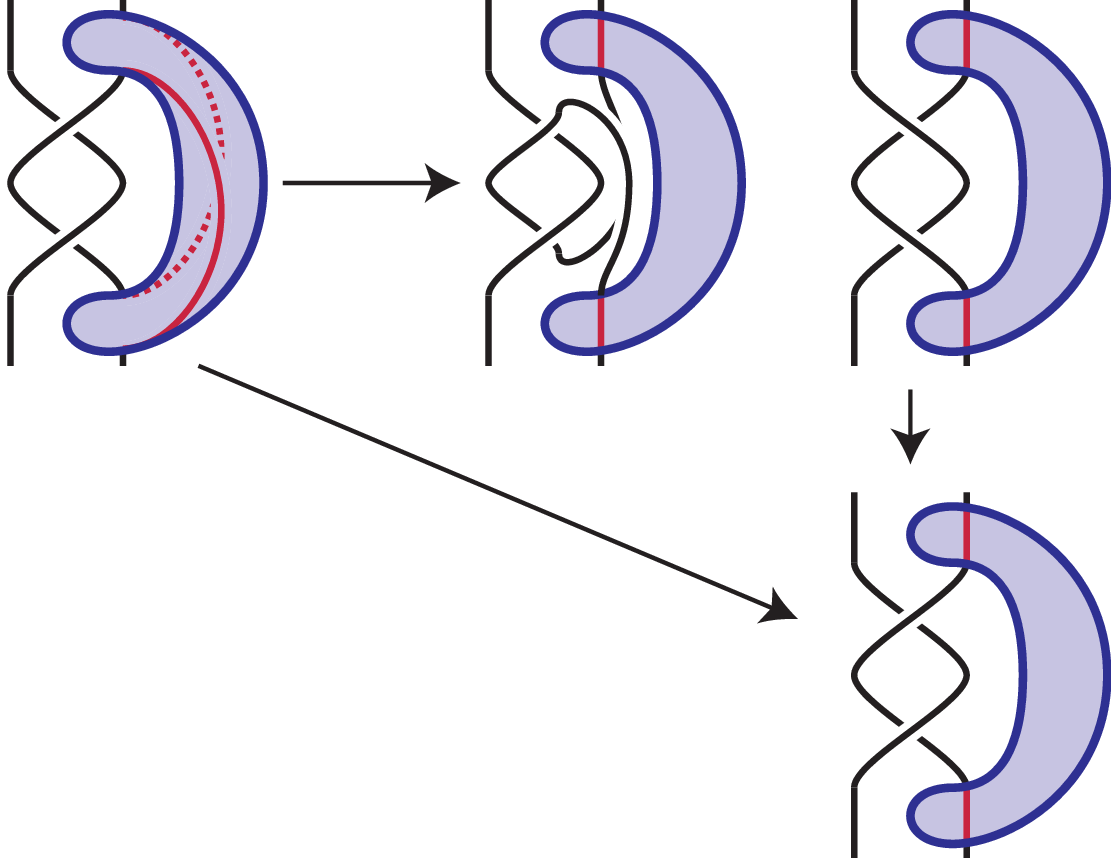}
\put(25.5,63.5){twisting}
\put(70,60){$=$}
\put(84,38){mirroring}
\put(10,28){orientation reversing}
\put(10,24){homeomorphism}
\put(10,20){taking the slopes $-1/1$ to $1/0$}
\end{overpic}
\caption{The banding ball is indicated by the blue ball.}
\label{fig:mundane}
\end{figure}
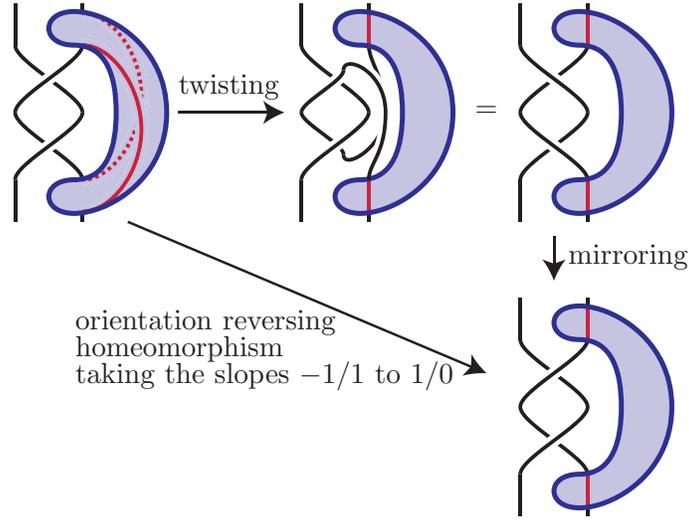

\section{Torus knot}\label{sec:TorusKnot}

In this section, we introduce an example of a non-trivial chirally cosmetic banding 
which does not realized by a $4$-move. 
Recently, Zekovi\'c~\cite{Zekovic} found that the $H(2)$-Gordian distance between 
the $(2,5)$-torus knot $T(2,5)$ and its mirror image $T(2,-5)$ is equal to one 
as shown in Figure~\ref{fig:5_1}. 
That is, $T(2,5)$ admits a cosmetic banding. 

First we show that this banding cannot be realized by a $4$-move. 
We denote by $\sigma(K)$ the signature of a knot $K$. 
If a knot $K_1$ is obtained from $K_0$ by a single crossing change, 
then we have $| \sigma(K_0) - \sigma(K_1)| \le 2$~\cite{Murasugi}. 
Since a $4$-move is realized by crossing changes two times, 
if a knot $K'$ is obtained from $K$ by a single $4$-move, 
then we have $| \sigma(K) - \sigma(K')| \le 4$. 
If a banding changes a knot $K$ into a knot $K'$, 
which is realized by a $4$-move as in the previous section, 
then $| \sigma(K) - \sigma(K')| \le 4$. 
On the other hand, we have $\sigma(T(2,5)) = -4$ and $\sigma(T(2,-5)) = 4$. 
Therefore this chirally cosmetic banding on $T(2,5)$ cannot be realized by a $4$-move. 

Next we consider the mechanism of this banding. 
Since the $(2,5)$-torus knot is the two-bridge knot of the form $S(5,1)$, 
we have a chirally cosmetic surgery on a knot in the lens space $L(5,1)$ 
by the Montesinos trick. 
Here we observe this phenomenon in detail. 

\begin{figure}[!htb]
\centering
\begin{overpic}[width=.75\textwidth]{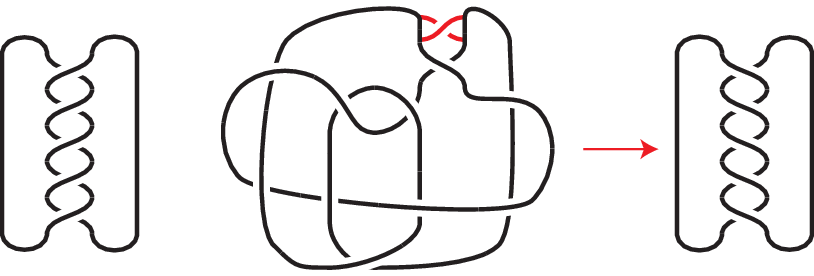}
\put(20,17){$=$}
\end{overpic}
\caption{$T(2,5)$ admits a cosmetic banding found by Zekovi\'c.}
\label{fig:5_1}
\end{figure}

By an isotopy, the cosmetic banding on $T(2,5)$ is shown as in Figure~\ref{fig:sibling1}. 
Decomposing to the tangles as in Figure~\ref{fig:sibling2}, 
we can draw the knot corresponding to the banding in $L(5,1)$ as in Figure~\ref{fig:sibling2}. 
The knot is actually the famous hyperbolic knot whose complement 
is called the ``figure-eight sibling''. 
It is known that the complement of the figure-eight sibling 
is amphicheiral~\cite{MartelliPetronio, Weeks}. 
Hence we can notice that the chirally cosmetic banding on $T(2,5)$ 
comes from the amphicheirality of the figure-eight sibling. 

\begin{figure}[!htb]
\centering
\begin{overpic}[width=.55\textwidth]{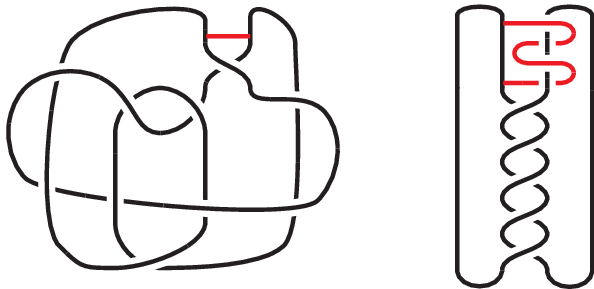}
\put(63,23){$=$}
\put(35,44){\Red{$-1$}}
\put(84.5,46){\Red{$-3$}}
\end{overpic}
\caption{A cosmetic bainding on $T(2,5)$}
\label{fig:sibling1}
\end{figure}

\begin{figure}[!htb]
\centering
\begin{overpic}[width=.65\textwidth]{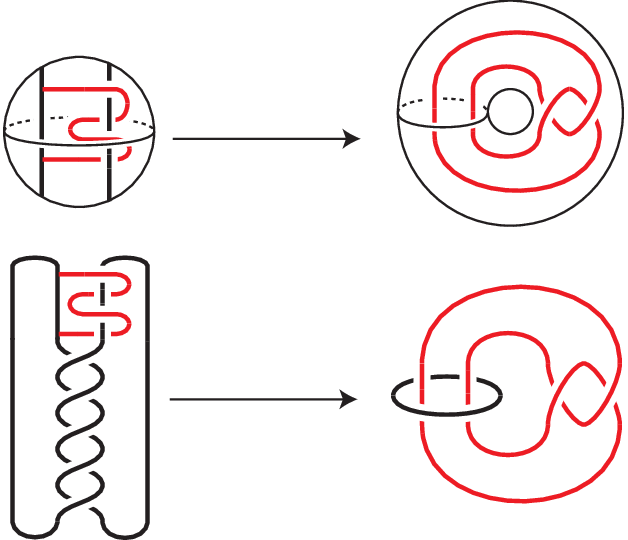}
\put(26,66.5){taking the double}
\put(26,59.5){branched cover}
\put(26,25){taking the double}
\put(26,18){branched cover}
\put(9,73){\Red{$-3$}}
\put(87,73){\Red{$-1$}}
\put(10,44){\Red{$-3$}}
\put(90,29){\Red{$-1$}}
\put(81,22){$5$}
\end{overpic}
\caption{The complement of the red colored knot in $L(5,1)$ 
is called the ``figure-eight sibling'' which is amphicheiral.}
\label{fig:sibling2}
\end{figure}

On the other hand, we can check that this chirally cosmetic banding is mundane as follows. 
This banding is described by the bandings with the slopes $0/1$ or $1/0$ as in Figure~\ref{fig:mundane5_1-1}. 
Therefore it suffices to show that the mirror image of the left side of Figure~\ref{fig:mundane5_1-1} with the slope $1/0$ is isotopic to the right side of Figure~\ref{fig:mundane5_1-1} with the same slope $1/0$. 
Such an isotopy is described in Figure~\ref{fig:mundane5_1-2}. 

\begin{figure}[!htb]
\centering
\begin{overpic}[width=.75\textwidth]{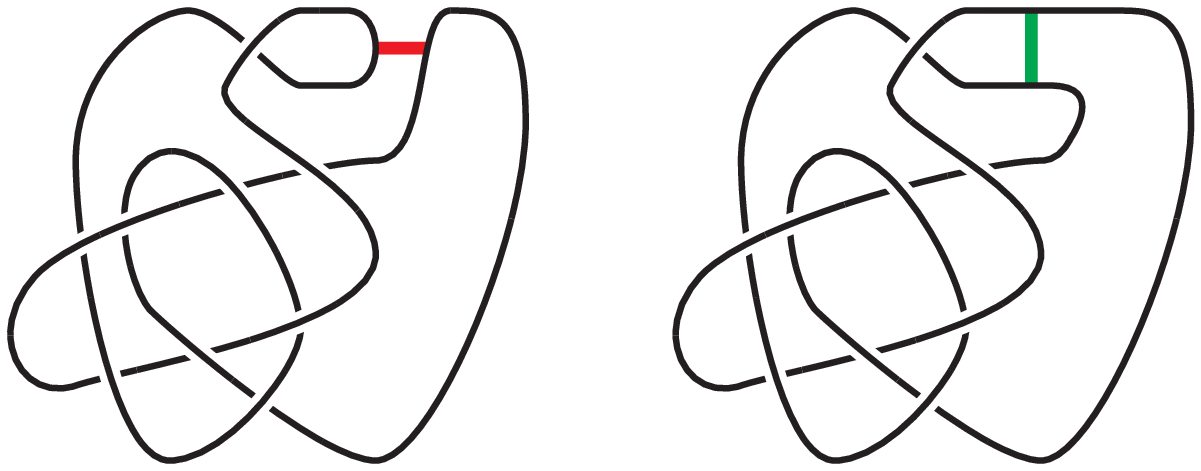}
\put(49,19){$=$}
\put(32.5,38){$\Red{\frac01}$}
\put(87,34.2){$\Green{\frac10}$}
\end{overpic}
\caption{The slope of the red banding is $0/1$ and that of the green banding is $1/0$.}
\label{fig:mundane5_1-1}
\end{figure}

\begin{figure}[!htb]
\centering
\begin{overpic}[width=\textwidth]{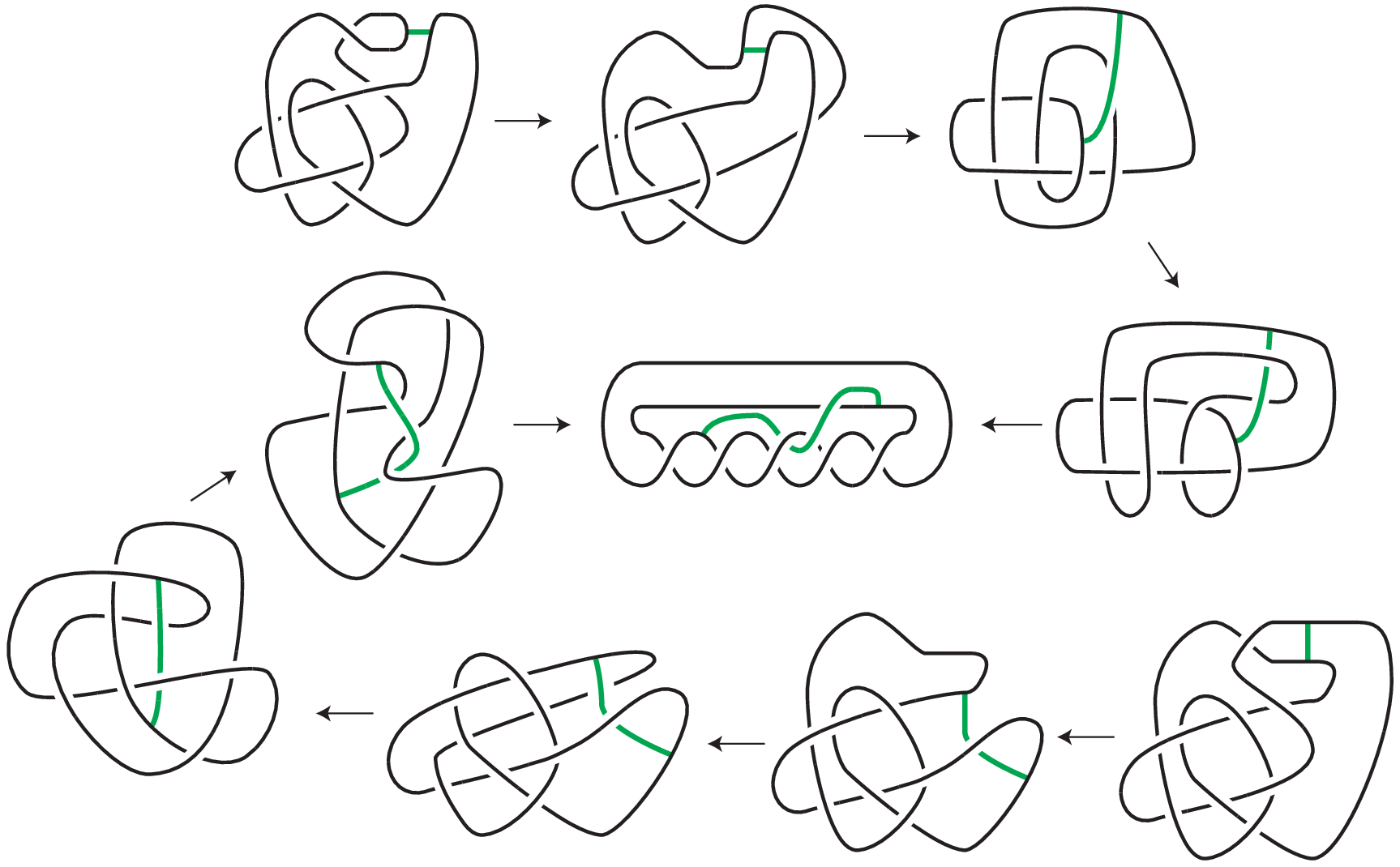}
%\put(49,19){$=$}
%\put(32.5,38){$\Red{\frac01}$}
%\put(87,34.2){$\Green{\frac10}$}
\end{overpic}
\caption{The slopes of the green bands are $1/0$.}
\label{fig:mundane5_1-2}
\end{figure}

\begin{remark}
As pointed out in~\cite{MartelliPetronio}, 
amphicheiral hyperbolic manifolds are quite sporadic, 
at least among those with small volume. 
Actually, the only amphicheiral hyperbolic manifolds with boundary 
obtained as fillings on the ``magic manifold'' are 
the figure-eight knot complement, its sibling, 
and the complement of the two-bridge link $S(10,3)$. 
The figure-eight sibling is the only example related to a cosmetic banding among them. 
\end{remark}

\begin{remark}
The two slopes of the cosmetic fillings on the figure-eight sibling are $1/0$ and $-1/1$. 
Hence the distance of the slopes is one. 
In other word, 
there exists a homeomorphism $h$ of the figure-eight sibling to its mirror image 
such that the distance between a slope $r$ and the slope $h(r)$ is just one. 
On the other hand, for an amphicheiral knot $K$ in $S^3$, 
the slope $p/q$ changed to $-p/q$ by the orientation reversing self-homeomorphism 
of the exterior of $K$. 
If $|p|, |q| \ne 0$, then the distance between the slopes $p/q$ and $-p/q$ is 
$2 |pq|$ which is grater than one. 
It follows that no more examples can be created from amphicheiral knots in $S^3$. 
In view of this, we ask the following. 
\end{remark}

\begin{question}
By generalizing this banding, that is, the cosmetic banding on $T(2,5)$, 
can we find other non-trivial cosmetic Dehn fillings 
on an amphicheiral cusped hyperbolic manifold along distance one slopes? 
\end{question}

\section{Cosmetic banding on the knot $9_{27}$ and its generalization}
\label{sec:9_27}

In this section, we first focus on the knot $9_{27}$ to show that it admits a cosmetic banding, and reveal its mechanism. 
As a matter of fact, the existence of this cosmetic banding comes from the famous example related to the Cosmetic Surgery Conjecture, 
which was given in \cite{BleilerHodgsonWeeks}. 
The mechanism for that example which we find enables us to construct infinitely many examples of knots and links admitting chirally cosmetic bandings. 
As a byproduct, we present a hyperbolic knot which admits exotic chirally cosmetic Dehn surgeries yielding hyperbolic 3-manifolds. 
This gives a counterexample to the conjecture raised by Bleiler, Hodgson and Weeks in \cite[Conjecture 2]{BleilerHodgsonWeeks}. 

\subsection{The knot $9_{27}$}

In \cite{BleilerHodgsonWeeks}, Bleiler, Hodgson and Weeks gave a hyperbolic knot in $S^2 \times S^1$ admitting a pair of exotic chirally cosmetic surgeries which yields the lens space $L(49,-19)$ and its mirror image. 
The knot is shown in Figure~\ref{fig:BHW}. 
We remark that the surgery slopes for the example are $19$ and $18$ with respect to a meridian-longitude system, and so, their distance is just one.

\begin{figure}[!htb]
\centering
\begin{overpic}[width=.35\textwidth]{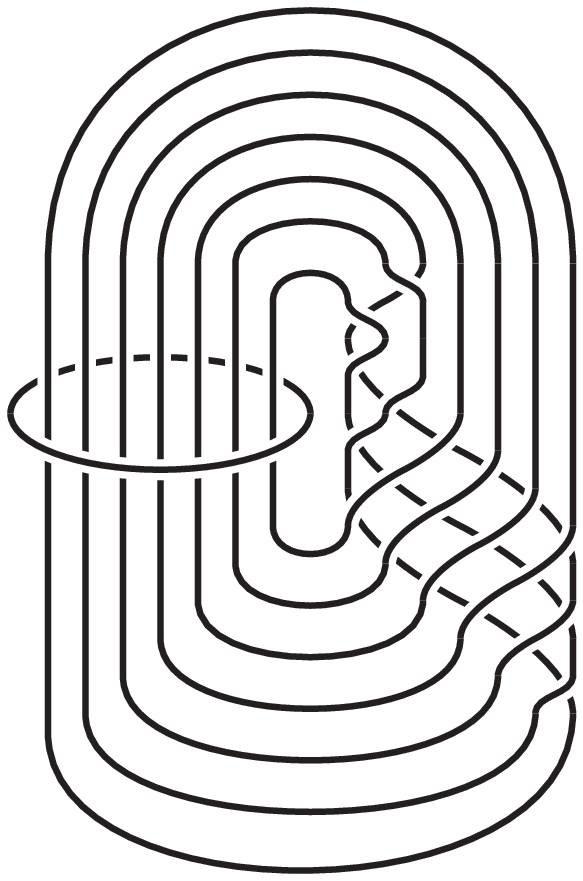}
\put(-4,52){$\dfrac01$}
\end{overpic}
\caption{A knot due to Bleiler, Hodgson and Weeks}
\label{fig:BHW}
\end{figure}

By considering the dual knot in $L(49,-19)$ for the knot in $S^2 \times S^1$, one can find a hyperbolic knot in $L(49,-19)$ which admits a non-trivial integral Dehn surgery yielding the mirror image of $L(49,-19)$, that is, $L(49,-18)$. 
Furthermore the knot in $L(49,-19)$ is strongly invertible, for the knot is obtained from the so-called Berge-Gabai knots in solid tori, and  all such knots have tunnel number one, that means the corresponding knots in $S^1 \times S^2$ are strongly invertible. 
See~\cite{BakerBuckLecuona} for detailed arguments for example. 

From this knot, by using the Montesinos trick, we have a chirally cosmetic banding on the knot $9_{27} = S(49,-19)$ as illustrated in Figure~\ref{fig:9_27}. 

\begin{figure}[!htb]
\centering
\begin{overpic}[width=.9\textwidth]{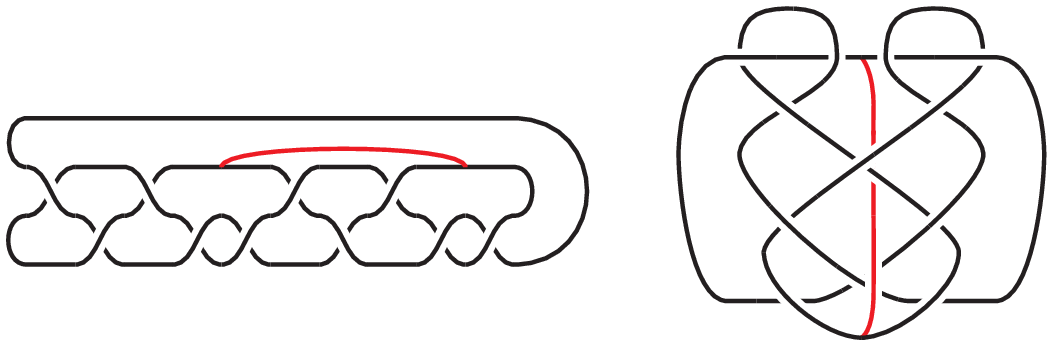}
\put(22,18.8){$\Red{1}$}
\put(80.6,22){$\Red{1}$}
\end{overpic}
\caption{The bandings along the red band with the slope $1$ are chirally cosmetic.}
\label{fig:9_27}
\end{figure}

Here we explain why such a cosmetic banding can occur for the knot $9_{27}$. 
The mechanism can be demonstrated in Figure \ref{fig-SymmetricPosition}. 
We first isotope the knot $9_{27}$ to the left position depicted in Figure \ref{fig-SymmetricPosition}, like a triangular prism. 
To the knot in the left, if we perform the banding along the red band in Figure~\ref{fig-SymmetricPosition} with the slope $0$, then we have the knot in the right. 
On the other hand, the knot in the right is actually the mirror image of the left. 
To see this, we perform mirroring the left knot along the middle horizontal plane shown in the figure, and then perform $2\pi/3$-rotation to obtain the right knot. 
Consequently, as shown in the figure, we can recognize that the knot $9_{27}$ can admit a chirally cosmetic banding. 

\begin{figure}[!htb]
\centering
\begin{overpic}[width=\textwidth]{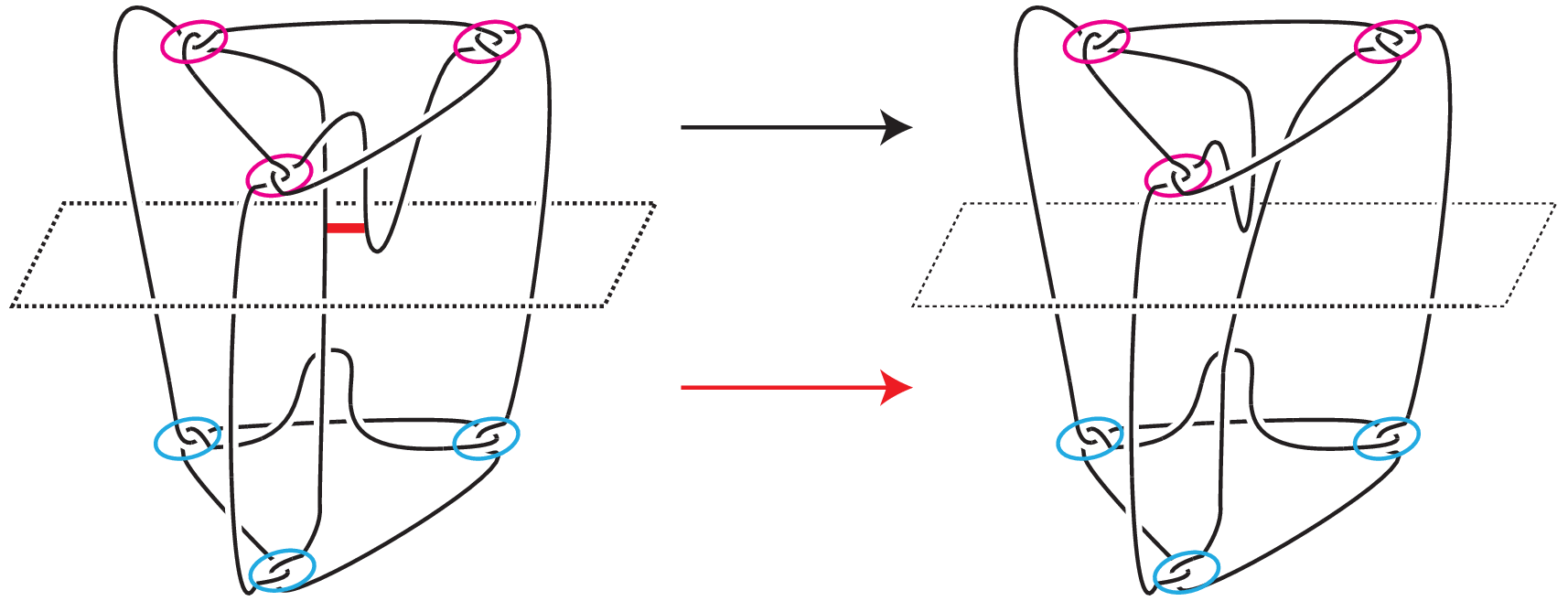}
\put(43.5,31.5){mirroring}
\put(43.5,27){and $\frac{2}{3}\pi$-rot.}
\put(43.5,15){banding}
\put(43.5,11){with slope $0$}
\end{overpic}
\caption{The mechanism of the chirally cosmetic banding.}\label{fig-SymmetricPosition}
\end{figure}
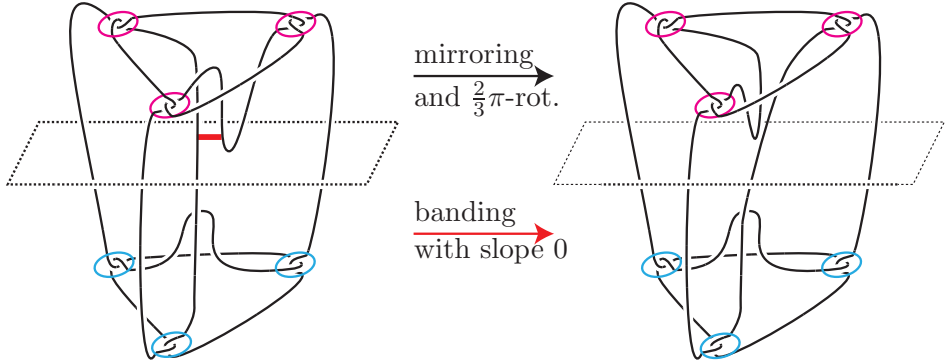

\subsection{Generalization}

Observing Figure \ref{fig-SymmetricPosition}, one can easily generalize it and have infinitely many examples of cosmetic bandings. 
For example, if one adds any number of twists to the single full-twists, or replaces the single full-twists by any rational tangles, in the knot shown in Figure \ref{fig-SymmetricPosition}, then one can obtain infinitely many examples of links admitting cosmetic bandings. 
By conversely using the Montesinos trick, we have infinitely many possibly new examples of knots admitting cosmetic surgeries. 

From now on, we focus on the knot $K$ which admits a cosmetic banding naturally extended from that on $9_{27}$ as illustrated in Figure~\ref{fig:pentagonal}. 
In some sense, it is obtained by the ``pentagonal version'' of the knot $9_{27}$. 
This knot $K$ admits a chirally cosmetic banding in the same way as seen in Figure~\ref{fig-SymmetricPosition}.

\begin{figure}[htb]
\includegraphics*[width=.4\textwidth]{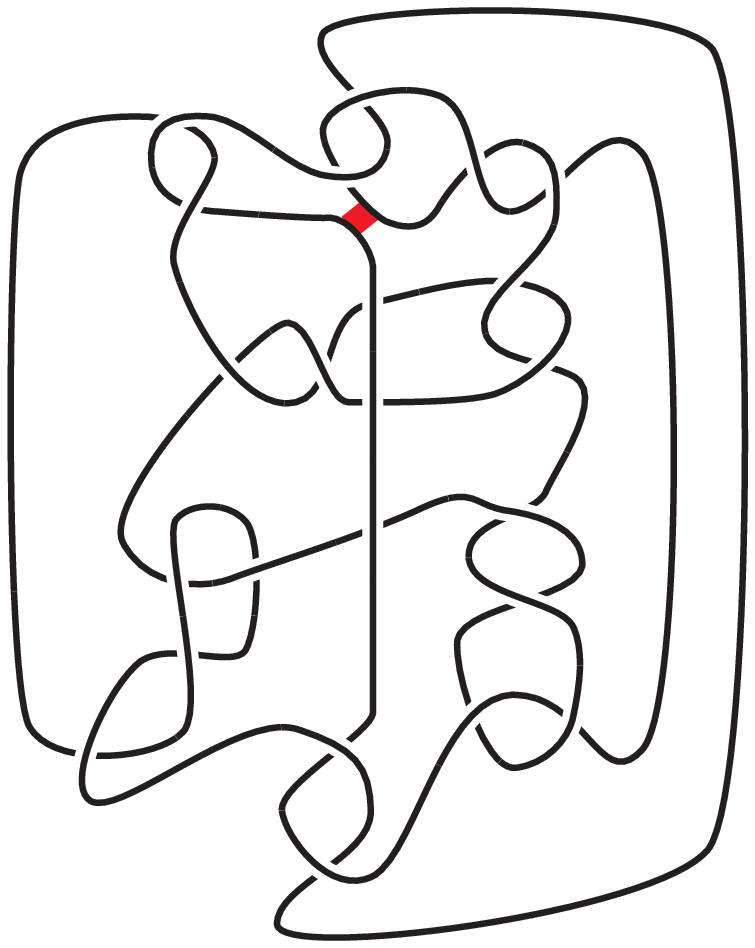}
\caption{The banding along the red band with the slope 0 is chirally cosmetic.}
\label{fig:pentagonal}
\end{figure}

Now let us consider the double branched cover $\bar{M}$ of $S^3$ branched along $K$. 
Applying the isotopic deformation as shown in Figure~\ref{fig:MontesinosTrick}, 
we can use the Montesinos trick conversely. 
Precisely, we can obtain a surgery description of $\bar{M}$ 
as in Figure~\ref{fig:SurgeryDescription} 
from the last picture of Figure~\ref{fig:MontesinosTrick}. 
Here we perform the meridional surgery on the red curve to obtain $\bar{M}$ in Figure~\ref{fig:SurgeryDescription} (i.e., we should ignore the red curve). 
The red curve actually appears as the preimage of the red band 
in Figure~\ref{fig:pentagonal}. 
Let $\bar{K}$ denote the red knot in $\bar{M}$ as shown in Figure \ref{fig:SurgeryDescription}. 

\begin{figure}[!htb]
\centering
\begin{overpic}[width=\textwidth]{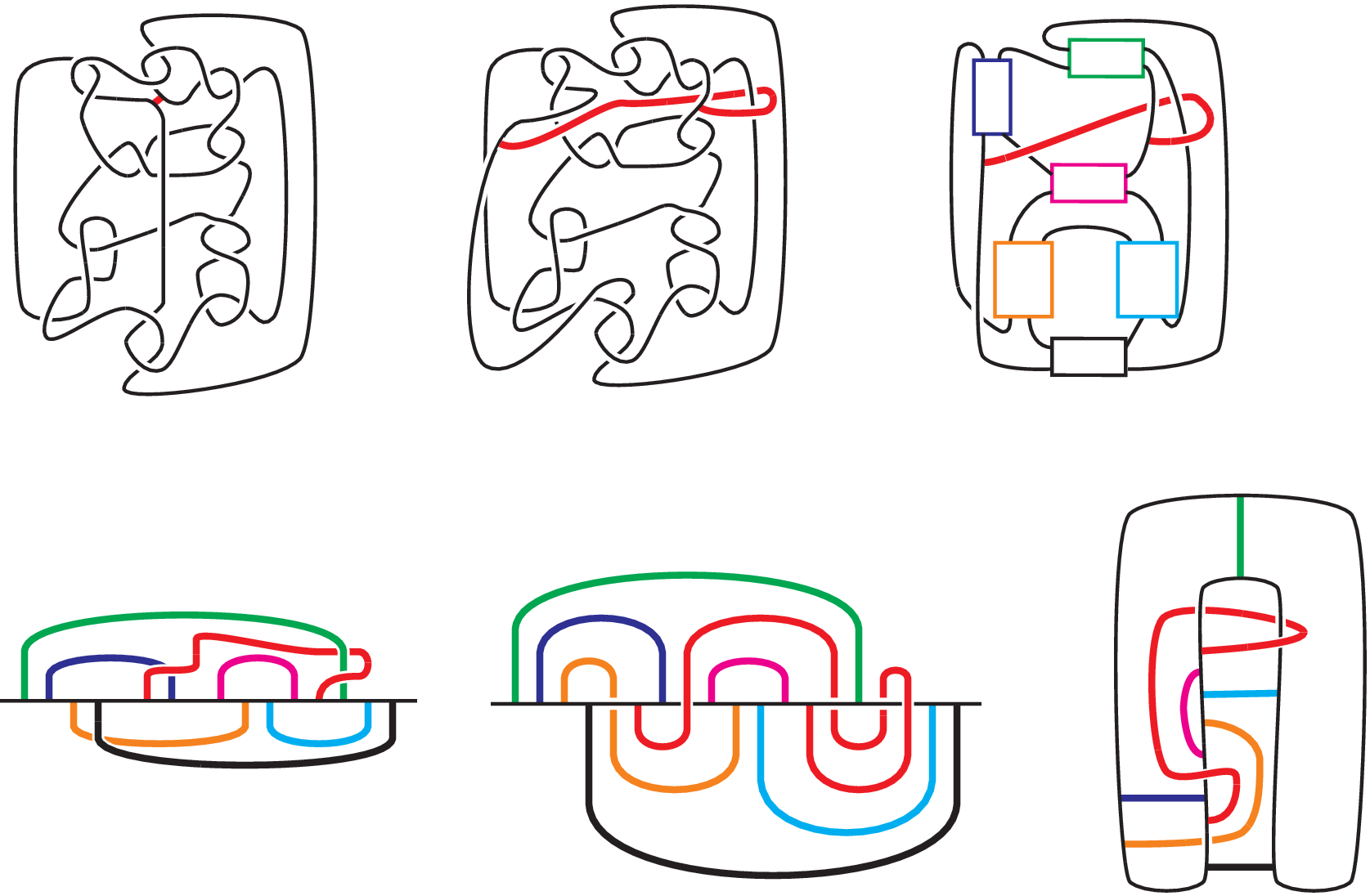}
\put(27,50){$=$}
\put(63,50){$=$}
\put(85,35){$\rotatebox{-60}{=}$}
\put(73,15){$\longleftarrow$}
\put(31,15){$=$}
\put(98.8,10){$\bullet$}
\put(100.5,10){$\infty$}
\put(73,17){cut}
\put(73,13){at $\infty$}
\put(9.7,55.7){\Red{$0$}}
\put(37,52){\Red{$1$}}
\put(77,56.2){\Red{$1$}}
\put(78.7,60.23){\Green{$3/2$}}
\put(71.8,57.3){\Blue{$3$}}
\put(77.4,51.2){\Magenta{$3/2$}}
\put(72.7,44){\Orange{$-\frac{3}{2}$}}
\put(81.8,44){\Cyan{$-\frac{3}{2}$}}
\put(77.6,38.5){$-3$}
\put(96,19){\Red{$1$}}
\put(91,25.5){$\Green{\frac32}$}
\put(80,7){\Blue{3}}
\put(84.5,15){\Magenta{$\frac32$}}
\put(78,3){\Orange{$-\frac{3}{2}$}}
\put(94,14){\Cyan{$-\frac{3}{2}$}}
\put(88,-0.3){$-2$}
\put(66,17){\Red{$1$}}
\put(50,25){$\Green{\frac32}$}
\put(46,17){\Blue{3}}
\put(57.5,16.2){\Magenta{$\frac32$}}
\put(48.,5){\Orange{$-\frac{3}{2}$}}
\put(53.5,3){\Cyan{$-\frac{3}{2}$}}
\put(66,1){$-2$}
\put(27.5,17){\Red{$3$}}
\put(16,22){$\Green{\frac32}$}
\put(7,15){\Blue{3}}
\put(16.8,14.8){\Magenta{$3/2$}}
\put(1,10){\Orange{$-\frac{1}{2}$}}
\put(20.1,12){\Cyan{$-3/2$}}
\put(12,7){$-2$}
\end{overpic}
\caption{The thick lines represent rational tangles.}
\label{fig:MontesinosTrick}
\end{figure}

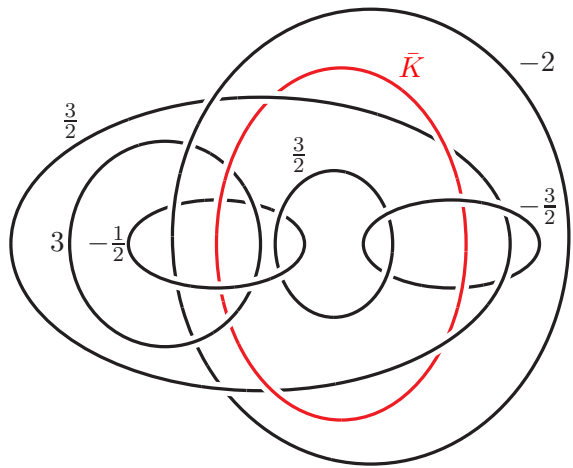
\begin{figure}[htb]
\centering
\begin{overpic}[width=.6\textwidth]{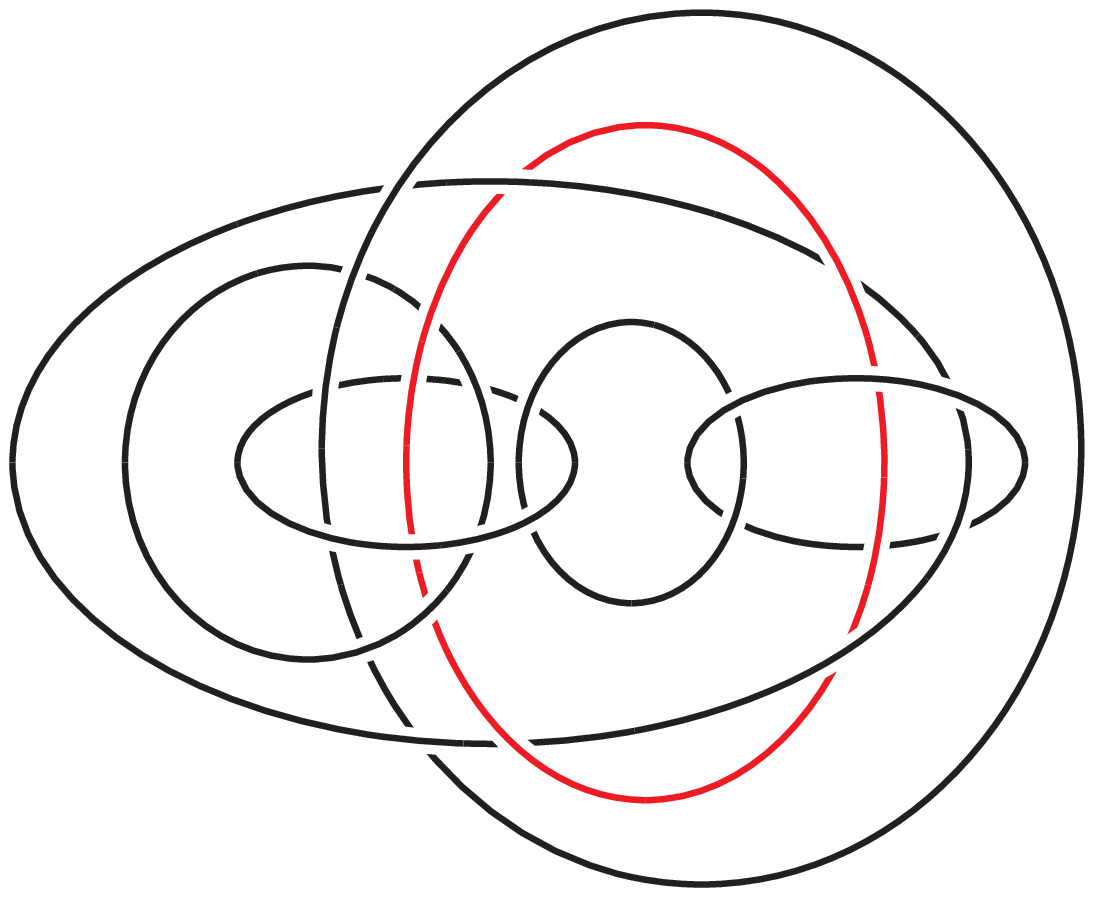}
%\put(70,68){$\Green{3}$}
\put(8,38.5){$3$}
\put(90,70){$-2$}
\put(90,45){$-\frac32$}
\put(10,60){$\frac32$}
\put(50,54){$\frac32$}
\put(14.5,38.5){$-\frac12$}
\put(69,69){$\Red{\bar{K}}$}
\end{overpic}
\caption{A counterexample of a conjecture proposed in \cite{BleilerHodgsonWeeks}.}
\label{fig:SurgeryDescription}
\end{figure}

By considering the knot $\bar{K}$ in $\bar{M}$, we can obtain the following. 

\begin{theorem}\label{thm;NewEx}
There exists a hyperbolic knot which admits exotic chirally cosmetic Dehn surgeries yielding hyperbolic 3-manifolds. 
Precisely, the hyperbolic knot given by the surgery description illustrated in Figure \ref{fig:SurgeryDescription}, 
which admits a pair of Dehn surgeries, one of which yields the mirror image of the 3-manifold obtained by the other, along inequivalent slopes. 
\end{theorem}

This gives a counterexample to the following conjecture. 
(We changed terminology slightly to fit ours.) 

\begin{conjecture}[{\cite[Conjecture 2]{BleilerHodgsonWeeks}}]
Any hyperbolic knots admit no cosmetic surgeries, purely or chirally, yielding hyperbolic manifolds.
\end{conjecture}

\begin{proof}[Proof of Theorem \ref{thm;NewEx}]
We consider a knot $\bar{K}$ in a 3-manifold $\bar{M}$ constructed above. 
Precisely it is described by the surgery description illustrated in Figure \ref{fig:SurgeryDescription}, where the red curve indicates the knot $\bar{K}$. 

The knot $\bar{K}$ admits a pair of cosmetic Dehn surgeries, one of which is the trivial surgery ($1/0$-surgery) yielding $\bar{M}$, 
and the other is the $3$-surgery yielding the mirror image of $\bar{M}$. 
By using the Montesinos trick, the corresponding knot $K$ in $S^3$ is shown in Figure \ref{fig:pentagonal}, 
and the knot $K$ admits a cosmetic banding in the same way as for the knot $9_{27}$. 
Also we can use \textit{SnapPy}, which is developed for studying the geometry and topology of 3-manifolds, to check that they are mutually homeomorphic orientation-reversingly. 

By using the computer program \emph{hikmot} developed in \cite{hikmot}, 
one can certify the manifold $\bar{M}$ is hyperbolic. 
The program utilizes interval arithmetics to solve the gluing equations, and can rigorously give a certification that a given manifold is hyperbolic. 

We next consider the complement $\bar{M} \setminus \bar{K}$ of the knot $\bar{K}$ in $\bar{M}$. 
Again one can certify the manifold $\bar{M} \setminus \bar{K}$ is hyperbolic by using the computer program \emph{hikmot}. 

To show that the slopes corresponding to $1/0$ and $3$ are inequivalent, due to \cite[Lemma 2]{BleilerHodgsonWeeks}, 
it suffices to show that the manifold $\bar{M} \setminus \bar{K}$ is chiral, 
that is, it is not orientation-preservingly homeomorphic to its mirror image. 
Actually this can be certified by using a computer-aided method developed in  \cite{DunfieldHoffmanLicata} based on \emph{hikmot} and \textit{SnapPy}.  

All the computations we performed are explained in Appendix~\ref{appA} in detail. This part is due to Hidetoshi Masai. 
\end{proof}

\subsection*{Acknowledgements}
The authors would like to thank Hidetoshi Masai for his helps to the computer-aided proof of Theorem \ref{thm;NewEx}. 

\appendix
\section{Computer aided proof}\label{appA}
\centerline{\textsc{by Hidetoshi Masai}}

\vspace{0.3cm}

We give a computer aided proof of the following theorem, which we need to complete the proof of Theorem \ref{thm;NewEx}.
\begin{theorem}
Let $\bar K$ denote the knot described in Figure \ref{fig:SurgeryDescription}.
Then we have, 
\begin{enumerate}
\item the knot $\bar K$ is hyperbolic,
\item the pair of manifolds $\bar K(1,0)$ and $\bar K(3,1)$ obtained by $(1/0)$-surgery and $3$-surgery on $\bar K$ respectively are hyperbolic,
\item $\bar K(1,0)$ is isometric to $\bar K(3,1)$, and 
\item the knot $\bar K$ is chiral.
\end{enumerate}
\end{theorem}
\begin{proof}
We use computer programs \textit{hikmot}, \textit{SnapPy}, and \texttt{canonical.py} developed in 
\cite{hikmot}, \cite{SnapPy}, and \cite{DunfieldHoffmanLicata} respectively.
All the codes and data we need for the proof are available at \cite{IJM}, see also README file in \cite{IJM}.
(1) and (2) are proved by \textit{hikmot}.
To apply \textit{hikmot}, we need to prepare triangulations with \textit{SnapPy}'s solution type ``all tetrahedra positively oriented''.
We call such a triangulation {\em positive}.
For (1), we can easily get a positive triangulation, and verify hyperbolicity by \textit{hikmot}.
However, after Dehn surgeries, we often get {\em non}-positive triangulations.
\textit{SnapPy} has a function called ``randomize'' to change triangulations randomly, 
and chances are that new triangulations are positive.
Unfortunately, it is often the case that after Dehn surgeries, even randomized triangulations can hardly be positive.
To solve this situation, we can try drilling short geodesics out. 
Then the trivial Dehn surgery of drilled manifold gives a new surgery description and hence a new triangulation which may be positive.
This is Algorithm 2 of \cite{hikmot}, but for completeness we rewrite the algorithm here.
\begin{algorithm}                      
\caption{Find a positive triangulation by drilling out}
\label{alg.drillout}
\begin{algorithmic}
\REQUIRE $M$ is a closed manifold with a surgery description.
\ENSURE $M$ has a positive triangulation.
	\WHILE{We can find a short closed geodesic $\gamma\subset M$}
			\STATE Drill out $\gamma$ to get $M\setminus\gamma$,
			\STATE Take filled\_triangulation $N$ of $M\setminus\gamma$,
			\STATE Fill the cusp of $N$ by the slope $(1,0)$.
			\STATE (By the above procedure, we forget the original surgery description and get a new surgery description.)
			\IF{$N$ has a positive triangulation.}
				\STATE return [True, $N$]
				\ENDIF
	\ENDWHILE
\STATE return False.
\end{algorithmic}
\end{algorithm}
This algorithm is implemented as \texttt{getpositive\_drill.py}. 
Two triangulation files \texttt{positive\_K10.tri} and \texttt{positive\_K31.tri} are positive triangulations
of $\bar K(1,0)$ and $\bar K(3,1)$ obtained by \texttt{getpositive\_drill.py}.
With these triangulations, we can verify the hyperbolicity of $\bar K(1,0)$ and $\bar K(3,1)$ by \textit{hikmot}.

For (3), we used \textit{SnapPy}'s \texttt{is\_isometric\_to} function, see also Remark \ref{rmk.falsepositive} below.

We now explain how we get (4).
The program \textit{hikmot} gives not only the verification of hyperbolicity, but also data of hyperbolic structures with rigorous error bounds.
Using the data, we can prove inequalities rigorously.
One application obtained in this manner is \texttt{canonical.py} due to Dunfield-Hoffman-Licata, which can rigorously ensure a given triangulation to be canonical in the sense of Epstein-Penner \cite{EP}.
Once we get the canonical triangulations, chirality can be checked from the combinatorial data of the canonical triangulations.
This has already been implemented on \textit{SnapPy} as \texttt{is\_amphicheiral} function on the symmetry group class that can be obtained by \texttt{symmetry\_group} function on the manifold class.
\end{proof}

\begin{remark}\label{rmk.falsepositive}
We used \textit{hikmot}, \textit{SnapPy}'s \texttt{is\_isometric} function, and \texttt{canonical.py} due to Dunfield-Hoffman-Licata.
They all never give false positives.
However the answer ``False'' of \textit{hikmot} and \texttt{canonical.py} only mean the failure of the verifications, and do not prove anything.
\textit{SnapPy}'s \texttt{is\_isometric} is rigorous provided we use triangulations which are verified to be canonical by \texttt{canonical.py}.
\end{remark}

\end{document}